\documentclass[12pt,dvips,twoside,letterpaper]{article}
\usepackage{pslatex}
\usepackage{fancyhdr}
\usepackage{graphicx}
\usepackage{geometry}
\RequirePackage[latin1]{inputenc}
\RequirePackage[T1]{fontenc}

\def\figurename{Figure} % Replace the colon that normally appears after the Figure number by a period.
\makeatletter
\renewcommand{\fnum@figure}[1]{\figurename~\thefigure.}
\makeatother
\def\tablename{Table} % Replace the colon that normally appears after the Figure number by a period.

\makeatletter
\renewcommand{\fnum@table}[1]{\tablename~\thetable.}
\makeatother
\usepackage{color}
                [2000/03/29 v1.0m Standard LaTeX package]

\usepackage{bbm}
\usepackage{amsmath}
\usepackage{amssymb}
\usepackage{amsfonts}
\usepackage{amsthm,amscd}
\newtheorem{theorem}{Theorem}[section]
\newtheorem{lemma}[theorem]{Lemma}

\newtheorem{proposition}[theorem]{Proposition}
\theoremstyle{definition}

\newtheorem{definition}[theorem]{Definition}

\theoremstyle{remark}
\newtheorem{remark}[theorem]{Remark}

\numberwithin{equation}{section}

\def\P{\mathbb P}
\def\R{\mathbb R}
\def\E{\mathbb E}

\def\Q{\mathbb Q}

\def\E{\mathbb E}
\def\N{\mathbb N}

\begin{document}

%------------------------------------------------------------------------------
\title{Stochastic viscosity solutions of reflected stochastic partial
differential equations with non-Lipschitz coefficients\thanks{This work is supported by the National Natural Science Foundation of China (N°11871076)}}

\author{Yong Ren$^{\,a,}$ \thanks{yongren@126.com}\;\; Jean Marc Owo$^{\,b,}$ \thanks{owo$_{-}$jm@yahoo.fr}\;\; and Auguste Aman$^{\,b,}$\thanks{augusteaman5@yahoo.fr, corresponding author}\\\\
a. Anhui Normal University, Department of Mathematics, Wuhu, Chine\;\;\; \;\;\;\;\;\;\;\;\;\;\;\;\;\;\;\;\;\;\;\;\;\;\;\;\\
b. Université Félix H. Boigny, UFR Mathématiques et Informatique,\;\; \;\;\;\;\;\;\;\;\;\;\;\;\;\; \;\;\;\;\;\;\;\;\;\;\;\;\;\; \;\\  Abidjan, C\^{o}te d'Ivoire\;\; \;\;\;\;\;\;\;\;\;\;\;\;
}

\date{}
\maketitle \thispagestyle{empty} \setcounter{page}{1}

% ------- [First Page Running Head] - place it immediately after title! ------
\thispagestyle{fancy} \fancyhead{}
 \fancyfoot{}
\renewcommand{\headrulewidth}{0pt}
%------------------------------------

\begin{abstract}
This paper, is an attempt to extend the notion of stochastic viscosity solution to reflected semi-linear stochastic partial
differential equations (RSPDEs, in short) with non-Lipschitz condition on the coefficients. Our method is fully probabilistic
and use the recently developed theory on reflected backward doubly stochastic differential equations (RBDSDEs,
in short). Among other, we prove the existence of the stochastic viscosity solution, and further extend the nonlinear
Feynman-Kac formula to reflected SPDEs, like one appear in \cite{2}. Indeed, in their recent work, Aman and Mrhardy \cite{2}
established a stochastic viscosity solution for semi-linear reflected SPDEs with nonlinear Neumann boundary condition
by using its connection with RBDSDEs. However, even Aman and Mrhardy consider a general class of reflected SPDEs,
all their coefficients are at least Lipschitz. Therefore, our current work can be thought of as a new generalization of a now
well-know Feymann-Kac formula to SPDEs with large class of coefficients, which does not seem to exist in the literature.
In other words, our work extends (in non boundary case) Aman and Mrhardy's paper.
\end{abstract}

\vspace{.08in}\textbf{MSC}:Primary: 60F05, 60H15; Secondary: 60J60\\ 
\vspace{.08in}\textbf{Keywords}: Stochastic viscosity solution, reflected backward doubly stochastic differential equations, non-Lipschitz
conditions.

\section{Introduction}
The notion of the viscosity solution for a partial differential equation, first introduced in 1983 by Crandall and Lions
[9], has had tremendous impact on the modern theoretical and applied mathematics. Today the theory has become an
indispensable tool in many applied fields, especially in optimal control theory and numerous subjects related to it. We
refer to the well-known "User's Guide" by Crandall et al. \cite{10} and the books by Bardi et al. \cite{5} and Fleming and Soner \cite{12} for a detailed account for the theory of (deterministic) viscosity solutions.
Given the importance of the theory, as well as the fact that almost all the deterministic problems in these applied fields
have their stochastic counterparts, it has long been desired that the notion of viscosity solution be extended to stochastic
partial differential equations; and consistent efforts have been made to prove or disprove such a possibility. Some of articles by Lions and Souganidis \cite{19,20} have finally shown an encouraging sign on this subject. Indeed, in \cite{19}, the
notion of stochastic viscosity solution was introduced for the first time; they use the so-called "stochastic characteristic" to
remove the stochastic integrals from a SPDEs, so that the stochastic viscosity solution can be studied $\omega$-wisely. They also,
in \cite{20}, derive the applications of such solutions to, among other things, pathwise stochastic control and front propagation
and phase transitions in random media were presented. Next, two others notions of stochastic viscosity solution of SPDEs
have been considered by Buckdahn and Ma respectively in \cite{6, 7} and \cite{8}. Roughly speaking, in \cite{6,7}, Buckdahn and Ma
consider a SPDE the following: for all $(t,x)\in [0,T]\times\R^n$ 
\begin{eqnarray*}
\left\{
\begin{array}{lll}
du(t,x) &=&\{Lu(t,x)+f(t,x,u(t,x),\sigma^*(x)Du(t,x))\}dt+\sum_{i=1}^{d}g_i(t,x,u(t,x))\circ \overleftarrow{dB}_t,\\	\\
u(T,x) & = & h(x),
\end{array}
\right.
\end{eqnarray*}
where $\circ \overleftarrow{dB}_t$ denote backward Stratonovich differential integral with a standard d-dimensional Brownian motion. The
function $f$, $g$, $h$ are measurable and $L$ is the second-order differential operator defined by:
where \begin{eqnarray}
L =\sum_{i,j=1}^{d}\sum_{l=1}^{k}\sigma_{il}(x)\sigma_{lj}(x)\partial_{x_{i}x_{j}}+\sum_{i=1}^{n}b_i(x)\partial_{x_i},\label{1.2}
\end{eqnarray}
in which $\sigma(.) = [\sigma_{ij}]^{n,k}_{i,j=1}, \; (b_1,\cdot\cdot\cdot, b_n)$ are certain measurable function and $\sigma^{*}(.)$ denotes the transpose of $\sigma(.)$. They
used the Doss-Sussman-type transformation (or the robust form). Although technically different, their method has the
same spirit as one appearing in \cite{19,20}. More precisely, they shown that under such a random transformation, SPDEs
can be converted to an ordinary PDE with random coefficients. Hence, they give a sensible definition of the stochastic
viscosity solution, which will coincide with the deterministic viscosity solution when f is deterministic and $g\equiv 0$. They
also naturally established the existence and uniqueness of this stochastic viscosity solution to SPDE. In \cite{8}, Buckdahn
and Ma show that an Itô-type random field with reasonably regular "integrands" can be expanded, up to the second
order, to the solutions to a fairly large class of stochastic differential equations with parameters, or even fully-nonlinear
stochastic partial differential equations, whenever they exist. Using such analysis they then propose a new definition of
stochastic viscosity solution for fully nonlinear stochastic PDEs, by the notion of stochastic sub and super jets in the spirit
of its deterministic counterpart. They also prove that this new definition is actually equivalent to the one proposed in
their previous works \cite{6} and  \cite{7}, at least for a class of quasilinear SPDEs. In all their previous three works, to establish existence and/or uniqueness of stochastic viscosity solution for SPDE, Buckdahn and Ma used the theory of backward
doubly SDEs introduced earlier by Pardoux and Peng  \cite{24} which is of the form
\begin{eqnarray}
Y_t = \xi+
\int_t^T f (s,Y_s,Z_s)ds+
\int_t^T g(s,Y_s,Z_s) \overleftarrow{dB}_s-\label{12}
\int_t^T Z_sdW_s
\end{eqnarray}
This kind of BDSDEs has a practical background, particularly in finance. In such domain, the extra noise $B$ can be regarded
as some extra information, which can not be detected in the financial market, but is available to some particular investors.
In their recent work, Aman and Mrhardy \cite{2} consider the following obstacle problem for SPDEs with nonlinear
Neumann boundary condition that we write formally as: for $\P$-a.e. $\omega\in\Omega$,
\begin{eqnarray*}
OP^{f,\phi,g,h,l}
\left\{
\begin{array}{ll}
(i) & \min\left\{u(t,x)-h(t,x),-\frac{\partial u(t,x)}{\partial t}-Lu(t,x)-f(t,x,u(t,x),\sigma^*(x)Du(t,x))\right.\\\\
&\left.-g(t,x,u(t,x)).\overleftarrow{\dot{B}}_t\right\}=0,\;\; (t,x)\in [0,T]\times \Theta,\\\\
(ii) & \frac{\partial u(t,x)}{\partial t}(t,x)+\phi(t,x,u(t,x))=0 \;\; (t,x)\in [0,T]\times \partial\Theta,\\\\
(iii) & u(T,x)=l(x),\;\; x\in\Theta,
\end{array}
\right.
\end{eqnarray*}

where $\Theta$ is a connected bounded domain included in $\R^d,\, (d\geq 1)$ and $f, l$ and $h$ are measurable functions. Finally 
$\overleftarrow{\dot{B}}_t=\frac{d\overleftarrow{\dot{B}}_t}{dt}$ is, at least formally, the time derivative of the standard Brownian motion $B$ called "white noise" and $g.\overleftarrow{\dot{B}}_t$ means the scalar product. They derived and proved a stochastic viscosity solution of this SPDE by a direct links with the
following reflected generalized BDSDE with Lipschitz coefficients: for all $t\in[0,T]$
\begin{eqnarray}
Y_t = \xi+\int_t^T f (s,Y_s,Z_s)ds+\int_t^T g(s,Y_s,Z_s) \overleftarrow{dB}_s+K_T-K_t-\int_t^T Z_sdW_s.\label{1.3}
\end{eqnarray} 
The increasing process $K$ is introduced to push the component $Y$ upwards so that it may remain above the given obstacle
process $S$. This push is minimal such that
\begin{eqnarray}
\int_t^T(Y_t-S_t)dK_t=0\label{flat}
\end{eqnarray}
which means that the push is only done when the constraint is saturated i.e. $Y_t = S_t$. In practice (finance market for example), the process $K$ can be regarded as the subsidy injected by a government in the market to allow the price process $Y$ of a commodity (cocoa, by example) to remain above a threshold price process $S$. Before Aman's work, Bahlali et al. \cite{4}) proved without application to reflected SPDE, existence and uniqueness result (resp. existence of minimal or maximal solution) of the previous RBDSDE when $\phi\equiv 0$, under global Lipschitz (resp. continuous) condition on the coefficient $f$. 
Unfortunately, the global Lipschitz or continuous condition cannot be satisfied in certain models that limits the scope of the result of Aman and Mrhardy \cite{2} for several applications (finance, stochastic control, stochastic games, SPDEs, etc,...). Some authors have previously tried to give weak conditions for reflected BDSDEs. We can cite the work of Aman \cite{1}, Aman and Owo \cite{3}. However, all this conditions remains insufficient to take into account all situations. For example, let consider the function $f$ and $g$ defined respectively by
\begin{eqnarray}\label{1.4}
f(t,y,z)=\frac{e^{-|y|}}{T^{1/4}}+\sqrt{\frac{C}{2}}z,\;\; g(t,y,z)=\frac{e^{-|y|}}{T^{1/4}}+\sqrt{\frac{\alpha}{2}}z
\end{eqnarray}
for any $C > 0$ and $0 < \alpha < 1$. It not difficult to prove that $f$ and $g$ are not Lipschitz and then the can not use the previous result to prove that RBDSDE with generator the function f and g defined in \eqref{1.4}.

To correct this shortcoming, we relax in this paper the global Lipschitz condition on the coefficients $f$ and $g$ to
following non-Lipschitz assumptions.\\

{\bf Main Assumptions}\\
There exist a non-random function $\rho:[0,T]\times\R^{+}\rightarrow \R^{+}$ which is not necessarily continuous in its first argument and satisfying "{\bf Condition A}", and two constants $C>0$ and $0< \alpha<1$ such that
\begin{eqnarray*}
\left\{
\begin{array}{lll}
|f(t,y_1,z_1)-f(t,y_2,z_2)|^2 &\leq &\rho(t,|y_1-y_2|^2)+C\|z_1-z_2\|^2\\\\
\|g(t,y_1,z_1)-g(t,y_2,z_2)\|^2 &\leq &\rho(t,|y_1-y_2|^2)+C\|z_1-z_2\|^2.
\end{array}
\right.
\end{eqnarray*}

{\bf Condition A}\\
For fixed $t\in[0,T],\; \rho(t,.)$ is continuous, concave and non-decreasing with $\rho(t,0) = 0$ such that:
\begin{enumerate}
\item [$(i)$] for fixed $u\in\R^{+}$,
\begin{eqnarray*}
\int_0^T\rho(t,u)dt<+\infty,
\end{eqnarray*}
[$(ii)$] for any $M>0$, if there exist a function $u:[0,T]\rightarrow \R^{+}$ solution of the following ordinary differential equation
\begin{eqnarray}\label{1.5}
\left\{
\begin{array}{lll}
u'(t)&=&-M\rho(t,u),\\\\
u(T)&=& 0.
\end{array}
\right.
\end{eqnarray}
then $u$ is unique and $u(t)\equiv 0,\;  t\in [0,T]$.  
\end{enumerate}
In this context, our paper have two goals:

First, we establish existence and uniqueness result for RBDSDE
\begin{eqnarray}
Y_t = \xi+\int_t^T f (s,Y_s,Z_s)ds+\int_t^T g(s,Y_s,Z_s) \overleftarrow{dB}_s+K_T-K_t-\int_t^T Z_sdW_s,\label{1.6}
\end{eqnarray}
when coefficients $f$ and $g$ satisfy "Main assumption" and hence establish a comparison principle. Next, using RBDSDE \eqref{1.6}, our second goal is to derive the
existence of a stochastic viscosity solution of SPDE $OP^{(f,0,g,h,l)}$ and further extend the nonlinear Feynman-Kac formula
in special case where the function $g$ does not depend on $z$. In our point of view, there exist real novelty in this work. Indeed, 
since functions $f$ and $g$ satisfy "Main assumptions", the "penalization method" that is usually used in the reflected BSDE
framework does not work. Consequently it is impossible to adapt the existing method to prove existence of a stochastic viscosity of $OP^{(f,0,g,h,l)}$ by a convergence result of a suitable sequence of non reflected SPDE $OP^{(f_n,0,g,h,l)}$, where for all $n\in\N$, the function $f_n$ defined by 
\begin{eqnarray*}
f_n(t,y, z)=f(t,y,z)-n(y-h(t,X_t ))^{-}	
\end{eqnarray*}
is obtained by penalization method (see \cite{15}, for more detail). For this reason, our method is based on the approximation
of function f by the sequence of Lipschitz function introduced in Lepeltier and San Martin \cite{LS}
The paper is organized as follows. In section 2, we give some notations and preliminaries, which will be useful
in the sequel. In section 3, we establish the existence and uniqueness theorem for a class of reflected BDSDEs with
non-Lipschitz coefficients.

\section{Reflected backward doubly stochastic differential equation with non Lipschitz coefficients}
\subsection{Preliminaries}
For a final time $T>0$, we consider $\{W_{t}; 0\leq t\leq
T \}$ and $\{B_{t}; 0\leq t\leq T \}$ two standard Brownian motion defined respectively on complete probability spaces $(\Omega_1,\mathcal{F}_1,\P_1)$ and $(\Omega_2,
\mathcal{F}_2,\P_2)$ with respectively $\mathbb{R}^{d}$ and $\mathbb{R}^{l}$ values. For any process $\{K_{t},\, t\in[0,T]\}$ defined on the completed probability space $\Omega_i,\mathcal{F}_i,\P_i)$ we set  the following family of $\sigma$-algebra $\mathcal{F}_{s,t}^{K}=\sigma\{K_{r}-K_{s},\, s\leq r \leq t \}$. In particular, $\mathcal{F}_{t}^{K}=\mathcal{F}_{0,t}^{K}$. Next, we consider the product space $(\Omega, \mathcal{F},\P)$, where
$$\Omega=\Omega_{1}\times\Omega_{2},\ \mathcal{F}=\mathcal{F}_{1}\otimes\mathcal{F}_{2}, \ \P=\P_{1}\otimes\P_{2}$$
and $\mathcal{F}_{t}=\mathcal{F}_{t}^{W} \otimes \mathcal{F}_{t,T}^{B}$. We should note that since $\textbf{\textit{F}}^{W}=(\mathcal{F}_{t}^W)_{t\in [0,T]}$ and $\textbf{\textit{F}}_{}^{B}=(\mathcal{F}_{t,T}^B)_{t\in [0,T]}$ are respectively increasing and decreasing filtration, the collection $\textbf{\textit{F}}_{}^{}=(\mathcal{F}_{t})_{t\in[0,T]}$ is neither increasing nor decreasing. Therefore it is not a filtration. Further, all random variables $\zeta$ and $\pi$ defined respectively in $\Omega_1$ and $\Omega_2$ are viewed as random variables on $\Omega$ via the following identification:
$$
\zeta(\omega)=\zeta(\omega_1);\;\;\;  \pi(\omega)=\pi(\omega_2), \;\;  \omega=(\omega_1,\omega_2).
$$
 We need in throughout this paper the following spaces:
 
$\mathcal{M}^{2}(\textbf{\textit{F}},[0,T];\mathbb{R}^{d\times k})$ denote the set of $d\P\otimes dt$ a.e. equal and $(d\times k)$-dimensional jointly measurable
random processes  $\{\varphi_{t}; 0\leq t\leq T \}$  such that
\begin{enumerate}
\item[(i)] $\displaystyle \|\varphi \|_{\mathcal{M}^{2}}^{2}=\mathbb{E}\left(\int_{0}^{T}\| \varphi_{t}
\|^{2} dt\right)< +\infty$

\item[(ii)] $\varphi_{t}$ is $\mathcal{F}_{t}$-measurable, for
a.e. $t \in [0,T].$
\end{enumerate}
We denote by  $\mathcal{S}^{2}(\textbf{\textit{F}},[0,T];\mathbb{R}^k)$
the set of continuous $k$-dimensional random processes such that
\begin{enumerate}
\item[(i)] $\|\varphi \|_{\mathcal{S}^{2}}^{2}=\mathbb{E}(\underset{0\leq t\leq T}
{\sup} \mid \varphi_{t}\mid^{2})< +\infty$

\item[(ii)] $\varphi_{t}$ is $\mathcal{F}_{t}$-measurable, for any
$t \in [0,T].$
\end{enumerate}
We denote also by $\mathcal{A}^{2}(\textbf{\textit{F}},[0,T];\mathbb{R})$ the set of continuous and increasing random processes $\{\varphi_{t}; 0\leq t\leq T \}$ such that
\begin{enumerate}
\item[(i)] $\|\varphi \|_{\mathcal{A}^{2}}^{2}=\mathbb{E}\left(| \varphi_{T}
|^{2}\right)< +\infty$

\item[(ii)] $\varphi_{t}$ is $\mathcal{F}_{t}$-measurable, for
a.e. $t\in[0,T]$.
\end{enumerate}
In the sequel, for simplicity, we shall set $\mathcal{S}^{2}(\R),\,\mathcal{M}^{2}(\R^{d})$ and $\mathcal{A}^{2}(\R^k)$ instead of \newline$\mathcal{S}^{2}(\bf{F},[0,T];\R^k)$, $\mathcal{M}^{2}({\bf F},[0,T],\R^{d})$ and $\mathcal{A}^{2}({\bf F},[0,T],\R^k)$ respectively and set $\mathcal{E}^{2}(0,T)=\mathcal{S}^{2}(\R)\times\mathcal{M}^{2}(\R^{d})\times\mathcal{A}^{2}(\R)$.
 
Let's give now our concept of solution that we will establish in the first part of this paper. 

\begin{definition}[Notion of solution]
\begin{enumerate}
\item [$(i)$] The triplet of processes $(Y,Z,K)$ is called solution of a RBDSDE \eqref{1.6} if it belongs in $\mathcal{E}^{2}(0,T)$ and satisfies satisfy \eqref{1.6} and \eqref{flat}.
\item[$(ii)$] The triplet of processes $(\underline{Y},\underline{Z},\underline{K})$ is said to be a minimal solution of a RBDSDE \eqref{1.6} if it belongs in $\mathcal{E}^{2}(0,T)$, satisfies satisfy \eqref{1.6} and for all any other solution $(Y,Z,K)$ of RBDSDE \eqref{1.6}, we have $\underline{Y}\leq Y$. 
\item [$(iii)$] The triplet of processes $(\overline{Y},\overline{Z},\overline{K})$ is said to be a maximal solution of a RBDSDE \eqref{1.6} if it belongs in $\mathcal{E}^{2}(0,T)$, satisfies satisfy \eqref{1.6} and that for all any other solution $(Y,Z,K)$ of RBDSDE \eqref{1.6}, we have $\overline{Y}\geq Y$. 
\end{enumerate}	
\end{definition}
All the result of the first part of our paper will be done under the following assumptions. The generators $f:\Omega \times [0,T]\times \R\times\R^d\rightarrow \R$ and $g:\Omega \times [0,T]\times \R\times\R^d\rightarrow \R^{\ell}$, the terminal value $\xi$ and the obstacle process $S=(S_t)_{t\geq 0}$ satisfy
\begin{description}
\item[$({\bf H1})$] $\xi$ is a $\mathcal{F}_T$-measurable random variable such that $\E(|\xi|^2)<+\infty$
\item [$({\bf H2})$] $S\in \mathcal{S}^{2}(\R)$ such that $S_T\leq \xi$ 
\item[$({\bf H3})$]  $f(.,y,z)$ and  $g(.,y,z)$ are jointly measurable such that $f (t,0,0)\in \mathcal{M}^2(\R)$ and $g(t,0,0)\in \mathcal{M}^2(\R^{\ell})$. Moreover for
all $(t,y_i,z_i)\in [0,T]]\times\R\times \R^{d},\; i=1,2$ we have:
\begin{eqnarray*}
\left\{
\begin{array}{lll}
|f(t,y_1,z_1)- f(t,y_2,z_2)|^{2}
&\leq & C(|y_1-y_2|^{2}+\|z_1-z_2\|^{2})\\\\
\|g(t,y_1,z_1)- g(t,y_2,z_2)\|^{2}
&\leq & C|y_1-y_2|^{2}+\alpha\|z_1-z_2\|^{2},
\end{array}
\right.
\end{eqnarray*}
where where $C > 0$ and $0<\alpha<1$.
\item[$({\bf H4})$]  $f(.,y,z)$ and  $g(.,y,z)$ are jointly measurable such that $f (t,0,0)\in \mathcal{M}^2(\R)$ and $g(t,0,0)\in \mathcal{M}^2(\R^{\ell})$. Moreover for
all $(t,y_i,z_i)\in [0,T]]\times\R\times \R^{d},\; i=1,2$ we have:
\begin{eqnarray*}
\left\{
\begin{array}{lll}
|f(t,y_1,z_1)- f(t,y_2,z_2)|^{2}
&\leq & \rho(t,|y_1-y_2|^{2})+C\|z_1-z_2\|^{2})\\\\
\|g(t,y_1,z_1)- g(t,y_2,z_2)\|^{2}
&\leq & \rho(t,|y_1-y_2|^{2})+\alpha\|z_1-z_2\|^{2},
\end{array}
\right.
\end{eqnarray*}
 where $C>0$ and $0<\alpha<1$ and $\rho:[0,T]\times\R^{+}\rightarrow\R^{+}$ is a non-random function satisfying "{\bf Condition A}". 
\end{description}
\begin{remark}
\begin{itemize}
\item[(i)] Lipschitz condition on generators $f,\, g$ with respect to the variable $y$ is the special case of $(\bf H3)$. It suffice to choose $\rho(t,u)=Cu$.
\item [(ii)] In addition to the case of Lipschitz, there exist these two following examples $\rho_1$ and $\rho_2$ defined by: for $\delta\in (0,1)$ be sufficiently small,
\begin{eqnarray*}
\rho_1(t,u)=\left\{
\begin{array}{ll}
u\log(u^{-1}),& 0\leq u\leq \delta,\\\\
\delta\log(\delta^{-1})+\kappa_1(\delta)(u-\delta),& u>\delta
\end{array}\right. 
\end{eqnarray*}
and 
\begin{eqnarray*}
\rho_2(t,u)=\left\{
\begin{array}{ll}
u\log(u^{-1})\log(\log(u)),& 0\leq u\leq \delta,\\\\
\delta\log(\delta^{-1})\log(\log(\delta))+\kappa_2(\delta)(u-\delta),& u>\delta,
\end{array}\right. 
\end{eqnarray*}
\end{itemize}
\end{remark}
Let recall some existence and uniqueness results establish previously by Bahali et al. \cite{4} under Lipschitz condition.
\begin{proposition}[Bahlali et al. \cite{4}]\label{P1}
Assume ${\bf(H1)}$-${\bf (H3)}$ hold. Then RBDSDEs \eqref{1.6} has a unique solution.	
\end{proposition}
\begin{proposition}[Aman and Owo \cite{3}]\label{l0a}
%Let $g$, $S^i$ and $\xi^i$ $(i=1,2)$ satisfy $(H1),\ (H2)$ and $(H4)$.
Assume that RBDSDEs associated respectively to $(f^1, g, \xi^1, S^1)$ and $(f^2, g, \xi^2, S^2)$ have solutions $(Y^1,Z^1, K^1)$ and $(Y^2,Z^2, K^2)$. Assume moreover that:

\begin{enumerate}
\item[$(i)$] $\xi^1\leq\xi^2$ a.s.,
\item[$(ii)$] $S_t^1\leq S_t^2$ a.s., for all
$t \in [0,T]$
\item[$(iii)$] $f^1$ satisfies ${\bf (H3)}$ such that $f^1(t,Y^2,Z^2)\leq f^2(t,Y^2,Z^2)$ a.s. (resp. $f^2$ satisfies ${\bf (H3)}$ such that $f^1(t,Y^1,Z^1)\leq f^2(t,Y^1,Z^1)$ a.s.).
 \end{enumerate}
 Then, $Y_t^1\leq Y_t^2$ a.s., for all $t\in[0,T]$.
\end{proposition}

\subsection{The main results}

Our objective, in this section is to derive an existence and uniqueness result for reflected BDSDEs with data $(\xi, f, g, S)$ where the generators are non-Lipschitz. More precisely we assume assumptions ${\bf(H1)}$, ${\bf(H2)}$ and ${\bf(H4)}$. For this purpose, let consider the sequence of processes $\left(
Y^n,Z^n,K^n\right)_{n\geq 1}$ defined recursively as follows. For $t\in[0,T],\; Y^0(t)=Z^0(t)=0$, and for all $\in \N^*$,
\begin{eqnarray}  \label{eq3}
\left\{
\begin{array}{ll}
&\displaystyle (i)\,Y_t^n=\xi+\int_{t}^{T}f(s,Y_s^{n-1},Z_s^{n})ds
\displaystyle+\int_{t}^{T}g(s,Y_s^{n-1},Z_s^n)\overleftarrow{dB}_{s}+\int_{t}^{T}dK_{s}^{n}-\int_{t}^{T}Z_s^ndW_{s},\\\\
&\displaystyle (ii)\,Y_{t}^n\geq S_{t},\\\\
&\displaystyle (iii)\,\int_0^T(Y_{t}^n-S_{t})dK_{t}^n=0  \hbox{.}
\end{array}
\right.
\end{eqnarray}
For each $n\geq 1$ and fixed $Y^{n-1}$, it not difficult to show that  the data of RBDSDEs \eqref{eq3} satisfy assumptions ${\bf(H1)}$, ${\bf(H2)}$ and ${\bf(H3)}$. Therefore, in view of Proposition \ref{P1}, RBSDEs \eqref{eq3} has a unique solution $ \left(Y^n,Z^n,K^n\right)_{n\geq 1}\in\mathcal{E}^{2}([0,T])$.

Our next aim is to prove that the sequence $\left(
Y^n,Z^n,K^n\right)_{n\geq 1}$ converges in $\mathcal{E}^{2}([0,T])$ to a process $(Y,Z,K)$ which is the unique solution of RBDSDEs \eqref{1.6}. We have this result.
\begin{theorem}\label{TE}
Assume that ${\bf(H1)}$, ${\bf(H2)}$ and ${\bf(H4)}$ hold.
Then the RBDSDEs \eqref{1.6} has a unique solution $(Y,Z,K) \in \mathcal{E}^{2}([0,T])$.
\end{theorem}
In order to provide the proof of Theorem \ref{te}, we need these two lemmas.
\begin{lemma}\label{l2}
Assume that ${\bf(H1)}$, ${\bf(H2)}$ and ${\bf(H4)}$ hold.
Then for all $0\leq t\leq T$, $n, m\geq 1$, we have
\begin{eqnarray*}
{\mathbb{E}}\left| Y_{t}^{n+m}-Y_{t}^{n}\right| ^{2} \leq e^{\frac{CT}{
1-\alpha}}\left(\frac{1-\alpha}{C}+1\right)\int_{t}^{T }\rho(s,{\mathbb{E}}
\left|Y_{s}^{n+m-1}-Y_{s}^{n-1}\right| ^{2})ds.
\end{eqnarray*}
\end{lemma}

\begin{proof} Using It\^o's formula, and the fact that $\displaystyle\int_{t}^{T }\left(
Y_{s}^{n+m}-Y_{s}^{n}\right)\left(dK_{s}^{n+m}-dK_{s}^{n}\right)\leq0$, we have
\begin{eqnarray*}
&&{\E}\left|Y_{t}^{n+m}-Y_{t}^{n}\right| ^{2}+{\E}\int_{t}^{T }\left|Z_{t}^{n+m}-Z_{t}^{n}\right| ^{2}ds \\
&\leq&2{\E}\int_{t}^{T }\left(
Y_{s}^{n+m}-Y_{s}^{n}\right)\left(f(s,Y_{s}^{n+m-1},Z_{s}^{n+m})-f(s,Y_{s}^{n-1},Z_{s}^{n})\right)
ds\\&&{\E}\int_{t}^{T }\left|
g(s,Y_{s}^{n+m-1},Z_{s}^{n+m})-g(s,Y_{s}^{n-1},Z_{s}^{n})\right|
^{2}ds .
\end{eqnarray*}
The rest of the proof follows as the one appear in N'zi and Owo \cite{MNJMO2}
\end{proof}

\begin{lemma}
\label{l3} Assume that ${\bf(H1)}$, ${\bf(H2)}$ and ${\bf(H4)}$ hold. Then, there exists $T_1 \in [0,T[$ and 
$M_1\geq 0$ such that for all $t\in[T_1, T]$ and $n\geq 1$, we have
${\mathbb{E}}\left|Y_{t}^{n}\right|^{2} \leq M_1$.
\end{lemma}

\begin{proof} Recall again It\^o's formula, we get
\begin{eqnarray*}
&&{\E}\left|Y_{t}^{n}\right| ^{2}+{\E}\int_{t}^{T }\left|Z_{t}^{n}\right| ^{2}ds \\
&=&{\E}\left|\xi\right| ^{2}+2{\E}\int_{t}^{T }\left\langle
Y_{s}^{n},f(s,Y_{s}^{n-1},Z_{s}^{n})\right\rangle
ds+2{\E}\int_{t}^{T }Y_{s}^{n}dK_{s}^{n}+{\E}\int_{t}^{T }\left| g(s,Y_{s}^{n-1},Z_{s}^{n})\right|
^{2}ds .
\end{eqnarray*}
Using ${\bf(H4)}$ and Young's inequality $2ab\leq
\frac{1}{\theta}a^2+\theta b^2$ for any $\theta>0$, we obtain
\begin{eqnarray}\label{f}
2\left\langle
Y_{s}^{n},f(s,Y_{s}^{n-1},Z_{s}^{n})\right\rangle &\leq&
\frac{1}{\theta}\left|Y_{s}^{n}\right|^2+\theta\left|f(s,Y_{s}^{n-1},Z_{s}^{n})\right|^2\nonumber\\
&\leq &\frac{1}{\theta}\left|Y_{s}^{n}\right|^2+
2\theta\rho(s,\left|Y_{s}^{n-1}\right|^2)+2\theta
C\|Z_{s}^{n}\|^2+2\theta\left|f(s,0,0)\right|^2,
\end{eqnarray}
and 
\begin{eqnarray}\label{g}
\left| g(s,Y_{s}^{n-1},Z_{s}^{n})\right|
^{2}\leq(1+\theta)\rho(s,\left|Y_{s}^{n-1}\right|^2)+
(1+\theta)\alpha\|Z_{s}^{n}\|^2+(1+\frac{1}{\theta})\left|g(s,0,0)\right|^2.	
\end{eqnarray}
Using again Young's inequality, we have for any $\beta>0$,
\begin{eqnarray*}
2\E\int_{t}^{T}Y_s^{n}dK_s^{n}
=2\E\int_{t}^{T}S_sdK_s^{n}
\leq\frac{1}{\beta}\E\underset{0\leq t\leq
T}{\sup}|S_s|^2+\beta\E\left(K_T^{n}-K_t^{n}\right)^2.
\end{eqnarray*}
But since
\begin{eqnarray*}
K_T^{n}-K_t^{n}&=&Y_t^{n}-\xi-
\int_{t}^{T}f(s,Y_s^{n-1},Z_s^{n})ds\\
&&-\int_{t}^{T}g(s,Y_s^{n-1},Z_s^{n})\overleftarrow{dB}_{s}+\int_{t}^{T}Z_s^{n}dW_{s},\ \ t \in [0,T],
\end{eqnarray*}
together with \eqref{f} and \eqref{g} lead for any $t\in [0,T]$,
\begin{eqnarray*}
&&\E\left(K_T^{n}-K_t^{n}\right)^2\\ &\leq &5\E\left(\left|Y_t^{n}\right|^2+|\xi|^2+
\left|\int_{t}^{T}f(s,Y_s^{n-1},Z_s^{n})ds\right|^2+\left|\int_{t}^{T}g(s,Y_s^{n-1},Z_s^{n})\overleftarrow{dB}_{s}\right|^2+
\left|\int_{t}^{T}Z_s^{n}dW_{s}\right|^2\right)\\
&\leq & 5\E\left(\left|Y_t^{n}\right|^2+|\xi|^2+
T\int_{t}^{T}\left(2\rho(s,\left|Y_{s}^{n-1}\right|^2)+2C\|Z_{s}^{n}\|^2+2\left|f(s,0,0)\right|^2\right) ds\right)\\
&&5\E\left(\int_{t}^{T}\left((1+\theta)\rho(s,\left|Y_{s}^{n-1}\right|^2)+
(1+\theta)\alpha\|Z_{s}^{n}\|^2+(1+\frac{1}{\theta})\left|g(s,0,0)\right|^2\right)ds+
\int_{t}^{T}\|Z_s^{n}\|^2 ds\right),
\end{eqnarray*}
Therefore,
\begin{eqnarray*}
&&\Big(1-5\beta\Big){\E}\left|Y_{t}^{n}\right| ^{2}+\Big[1-2\theta
C-(1+\theta)\alpha-5\beta(2TC+(1+\theta)\alpha+1)\Big]{\E}\int_{t}^{T
}\left|Z_{s}^{n}\right|^2 ds \\
&\leq&\Big(1+5\beta\Big){\E}\left|\xi\right| ^{2}+\frac{1}{\theta}{\E}\int_{t}^{T
}\left|Y_{s}^{n}\right|^2 ds+\Big[(3\theta+1)+5\beta(1+\theta)+10\beta T\Big]\int_{t}^{T }
\rho(s,{\E}\left|Y_{s}^{n-1}\right|^2) ds\\&&+{\E}\int_{t}^{T
}\Big[(2\theta+10\beta T)\left|f(s,0,0)\right|^2+
(1+\frac{1}{\theta})(1+5\beta)\left|g(s,0,0)\right|^2\Big]ds+\frac{1}{\beta}\E\underset{0\leq t\leq
T}{\sup}|S_s|^2.
\end{eqnarray*}
Choosing $\beta,\theta>0$ such that $\beta<\displaystyle\frac{1-\alpha}{5(2TC+\alpha+1)}$ and $\theta\leq\displaystyle\frac{1-\alpha-5\beta(2TC+\alpha+1)}{2C+\alpha+5\beta\alpha}$, there exists a constant $c=c(\alpha,T,C)>0$ satisfying
\begin{eqnarray*}
{\E}\left|Y_{t}^{n}\right| ^{2}
&\leq&c+c{\E}\left|\xi\right| ^{2}+c{\E}\int_{t}^{T
}\left|Y_{s}^{n}\right|^2 ds+c\int_{t}^{T }
\rho(s,{\E}\left|Y_{s}^{n-1}\right|^2) ds\\&&+c{\E}\int_{t}^{T
}\Big[\left|f(s,0,0)\right|^2+
\left|g(s,0,0)\right|^2\Big]ds+c\E\underset{0\leq t\leq
T}{\sup}|S_s|^2.
\end{eqnarray*}
Hence, it follows from Gronwall's inequality that
\begin{eqnarray}\label{g1}
{\E}\left|Y_{t}^{n}\right| ^{2} &\leq&\mu_t^1+ce^{cT}\int_{t}^{T }
\rho(s,{\E}\left|Y_{s}^{n-1}\right|^2) ds
\end{eqnarray}
where
\begin{eqnarray*}
\mu_t^1=ce^{cT}\left(1+{\E}\left|\xi\right|
^{2}+{\E}\underset{0\leq t\leq
T}{\sup}|S_s|^2+{\E}\int_{t}^{T
}\left[\left|f(s,0,0)\right|^2\!+\!
\left|g(s,0,0)\right|^2\right]ds\right).	
\end{eqnarray*}
Let set $M=\max\left\{ce^{cT},;\left(\frac{1-\alpha}{C}+1\right)
e^{\frac{CT}{1-\alpha}}\right\}$ and $M_1=2\mu_0^1$. Recall $(i)$ of {\bf Condition A}, we have $\int_{0}^{T } \rho(s,M_1) ds<+\infty$ and hence there exists $T_1\in [0,T]$ such that
$\int_{T_1}^{T } \rho(s,M_1) ds=\frac{\mu_0^1}{M}$. If $\int_{0}^{T } \rho(s,M_1) ds=\frac{\mu_0^1}{M}$ then $T_1=0$. But if   $\int_{0}^{T } \rho(s,M_1) ds>\frac{\mu_0^1}{M}$, so for all $t\in [T_1,T]$ it follows from \eqref{g1}, the fact that $\rho(t,.)$ is non-decreasing and the induction method that  
${\E}\left|Y_{t}^{n}\right| ^{2} \leq M_1,$ for all $n\geq 1$.
\end{proof}
Now we are able to give the proof of Theorem \ref{te}.
\begin{proof}[Proof of Theorem \ref{te}]

{\bf Existence}\\ For any $t\in [0,T]$, we consider the sequence of processes $(\phi_n(t))_{n\geq 1}$ defined recursively by
\begin{eqnarray*}
\phi_0(t)=M\int_{t}^{T } \rho(s,M_1)ds \ \ \text{and}\ \
\phi_{n+1}(t)=M\int_{t}^{T } \rho(s,\phi_{n}(s)) ds.	
\end{eqnarray*}
With the same reasons as those given in \cite{MNJMO2}, $(\phi_n(t))_{n\geq 0}$ is a non-increasing sequence and converges uniformly to $0$ for all $t\in [T_1,T]$. Moreover, Lemmas \ref{l2} and \ref{l3} permit us to derive that for all $t\in[T_1,T]$ and $n,\ m \geq 1$
\begin{eqnarray}\label{3i}
{\E}\left|
Y_{t}^{n+m}-Y_{t}^{n}\right| ^{2} \leq\phi_{n-1}(t)\leq M_1.
\end{eqnarray}
On the other hand, It\^o's formula together with the fact that
\begin{eqnarray*}
\int_{t}^{T }\left(
Y_{s}^{n+m}-Y_{s}^{n}\right)\left(dK_{s}^{n+m}-dK_{s}^{n}\right)\leq 0,	
\end{eqnarray*}
assumptions $({\bf H1})$ and  $({\bf H4})$ and Young's inequality
$2ab\leq \frac{1}{\theta}a^2+\theta b^2,\ \theta>0$ lead that for all $t\in [T_1,T]$
\begin{eqnarray*}
&&\left|Y_{t}^{n+m}-Y_{t}^{n}\right|^{2}+(1-\theta
C-\alpha)\int_{t}^{T }\left|Z_{t}^{n+m}-Z_{t}^{n}\right| ^{2}ds\\
&\leq&\frac{1}{\theta}\int_{t}^{T
}\left|Y_{s}^{n+m}-Y_{s}^{n}\right|^2 ds+(\theta+1)\int_{t}^{T
} \rho(s,\left|Y_{s}^{n+m-1}-Y_{s}^{n-1}\right|^2) ds\\
&&+2\int_{t}^{T }\left\langle
Y_{s}^{n+m}-Y_{s}^{n},(g(s,Y_{s}^{n+m-1},Z_{s}^{n+m})-g(s,Y_{s}^{n-1},Z_{s}^{n}))
\overleftarrow{dB_{s}}\right\rangle\\
&&-2\int_{t}^{T }\left\langle
Y_{s}^{n+m}-Y_{s}^{n},(Z_{s}^{n+m}-Z_{s}^{n})
dW_{s}\right\rangle.
\end{eqnarray*}
Furthermore, setting $\theta=\frac{1-\alpha}{2C}$ with no more difficult calculations and \eqref{3i} we obtain
\begin{eqnarray*}
\sup_{T_1\leq t\leq T }\left({\E}\left|Y_{t}^{n+m}-Y_{t}^{n}\right| ^{2}\right)+\frac{1-\alpha}{2}{\E}\int_{T_1}^{T }\left|Z_{t}^{n+m}-Z_{t}^{n}\right| ^{2}ds
&\leq&\left(\frac{T-T_1}{\theta}+\frac{\theta+1}{M}\right)\phi_{n-1}(T_1).
\end{eqnarray*}
from which, we deduce by Burkhölder-Davis-Gundy's inequality that
\begin{eqnarray*}
{\E}\sup_{T_1\leq t\leq T }\left|Y_{t}^{n+m}-Y_{t}^{n}\right| ^{2}+
{\E}\int_{T_1}^{T }\left|Z_{t}^{n+m}-Z_{t}^{n}\right| ^{2}ds
\leq \lambda\phi_{n-1}(T_1),
\end{eqnarray*}
where $\lambda$ is positive constant depending on
$C$, $T_1$, $T$, $\alpha$ and $M$. Since $\phi_{n}(t)\ \rightarrow \ 0$,\ for all $t\in[T_1,T]$, as
$n \ \rightarrow \ \infty$, it follows that $(Y^n,Z^n)$ is a Cauchy
sequence in the Banach space $\mathcal{S}^{2}([T_1,T])\times
\mathcal{M}^{2}([T_1 ,T])$. Therefore it converges to a process $(Y,Z)$ belonging in $\mathcal{S}^{2}([T_1,T])\times
\mathcal{M}^{2}(T_1,T])$. On other words, we have 
\begin{eqnarray*}
{\E}\int_{T_1}^{T}\left|Z_{s}^{n}-Z_{s}^{}\right| ^{2}ds\ \longrightarrow \ 0,\; \; \; as\; \;
n \ \longrightarrow \ \infty	
\end{eqnarray*}
and 
\begin{eqnarray*}\label{}
{\E}\left| Y_{t}^{n}-Y_{t}\right|
^{2}\longrightarrow \ 0,\; \; \; \mbox{as}\; \;
n \ \longrightarrow \ \infty.
\end{eqnarray*}
Next, applying Hölder, BDG's inequalities and $({\bf H4})$ we respectively
\begin{eqnarray}\label{bu1}
&&{\E}\left|\int_{t}^{T }
(f(s,Y_s^n,Z_s^n)-f(s,Y_s,Z_s))ds\right|^2\notag\\&&\leq
(T-T_1)C{\E}\int_{T_1}^{T }\left|Z_{s}^{n}-Z_{s}^{}\right|^2
ds+(T-T_1){\E}\int_{T_1}^{T }
\rho(s,\left|Y_{s}^{n}-Y_{s}^{}\right|^2) ds\notag\\&&\leq
(T-T_1)C{\E}\int_{T_1}^{T }\left|Z_{s}^{n}-Z_{s}^{}\right|^2
ds+\frac{(T-T_1)}{M}\phi_{n}(T_1),
\end{eqnarray}

\begin{eqnarray}\label{bu2}
&&{\E}\underset{T_1\leq t\leq T}\sup\left|\int_{t}^{T }
g(s,Y_s^n,Z_s^n)dB_s-\int_{t}^{T }g(s,Y_s,Z_s)dB_s\right|^2\notag\hspace{3cm}\\&&\leq \alpha{\E}\int_{T_1}^{T
}\left|Z_{s}^{n}-Z_{s}^{}\right|^2 ds+\frac{1}{M}\phi_{n}(T_1).
\end{eqnarray}
and
\begin{eqnarray}\label{bu3}
{\E}\underset{T_1\leq t\leq T}\sup\left|\int_{t}^{T }
Z_s^ndW_s-\int_{t}^{T }Z_sdW_s\right|^2\leq {\E}\int_{T_1}^{T
}\left|Z_{s}^{n}-Z_{s}^{}\right|^2 ds
\end{eqnarray}
Therefore according to above, we have for all
$t\in[T_1,T]$,
\begin{eqnarray*}
\int_{t}^{T}f(s,Y_s^n,Z_s^n)ds \longrightarrow \int_{t}^{T
}f(s,Y_s,Z_s)ds \ \ \text{\ in \ } \mathbb{P}-\text{probability},
\ \ \text{\ as \ } n\rightarrow \infty,	
\end{eqnarray*}
\begin{eqnarray*}
\int_{t}^{T }g(s,Y_s^n,Z_s^n)dB_s \longrightarrow \int_{t}^{T
}g(s,Y_s,Z_s)dB_s \ \ \text{\ in \ }
\mathbb{P}-\text{probability}, \ \ \text{\ as \ } n\rightarrow
\infty,	
\end{eqnarray*}
and
\begin{eqnarray*}
\int_{t}^{T } Z_{s}^{n}dW_s \longrightarrow \int_{t}^{T }
Z_{s}dW_s\ \ \text{\ in \ } \mathbb{P}-\text{probability}, \ \
\text{\ as \ } n\rightarrow \infty.
\end{eqnarray*}
On the other hand in view of \eqref{eq3}, we get also
\begin{eqnarray*}
\E\underset{T_1\leq t\leq T}\sup\left|K_{t}^{m+n}-K_{t}^{n}\right|^2
&\leq&\E|Y_{T_1}^{m+n}-Y_{T_1}^{n}|^2+\E\underset{T_1\leq t\leq
T}\sup\left|Y_{t}^{m+n}-Y_{t}^{n}\right|^2\\
&&+
\E\left|\int_{T_1}^{T}\left(f(s,Y_{s}^{m+n-1},Y_{s}^{m+n})
-f(s,Y_{s}^{n-1},Z_{s}^n)\right)ds\right|^2\\
&&+\E\underset{T_1\leq t\leq T}\sup\left|\int_{T_1}^{t}\left(g(s,Y_{s}^{m+n-1},Y_{s}^{m+n})
-g(s,Y_{s}^{n-1},Z_{s}^n)\right)\overleftarrow{dB}_{s}\right|^2\\
&&+\E\underset{T_1\leq t\leq T}\sup\left|\int_{T_1}^{t}(Z_{s}^{m+n}-Z_{s}^n)dW_{s}\right|^2,
\end{eqnarray*}
which provides according to \eqref{bu1}, \eqref{bu2} and \eqref{bu3}
\begin{eqnarray*}
\E\underset{T_1\leq t\leq T}\sup\left|K_{t}^{m+n}-K_{t}^{n}\right|^2\rightarrow 0,\; \; \; as\; \; n \rightarrow\infty.
\end{eqnarray*}
So, there exists a $\mathcal{F}_t$-measurable process $K$ with value in $\R_+$ such that
\begin{eqnarray*}
\E\underset{T_1\leq t\leq T}\sup\left|K_{t}^{n}-K_{t}\right|^2\longrightarrow 0,\; \; \; as\; \; n \rightarrow\infty.
\end{eqnarray*}
Obviously, $\{K_{t};\ T_1\leq t \leq T\}$ is a non-decreasing and continuous process. Passing to the limit in $(i)$ and $(ii)$ of \eqref{eq3}, we have for any $t\in[T_1,T]$,
\begin{eqnarray*} \label{}
\left\{
\begin{array}{ll}
&\displaystyle (i)\,Y_t=\xi+\int_{t}^{T}f(s,Y_s,Z_s)ds
\displaystyle+\int_{t}^{T}g(s,Y_s,Z_s)\overleftarrow{dB}_{s}+\int_{t}^{T}dK_{s}-\int_{t}^{T}Z_sdW_{s},\\\\
&\displaystyle (ii)\,Y_{t}\geq S_{t}.
\end{array}
\right.
\end{eqnarray*}
It remain to prove \eqref{flat}. For that, it follows from Saisho \cite{Saisho} (see p. 465) that
\begin{eqnarray*}
\int_{0}^T(Y_{s}^n-S_{s})\mathbf{1}_{[T_1,T]}dK_{s}^n\ \rightarrow\int_0^T(Y_{s}-S_{s})\mathbf{1}_{[T_1,T]}dK_{s}\; \; \; \P-a.s.,\; \; \; as \; \; n\rightarrow\infty.
\end{eqnarray*}
Now, according to $(iii)$ of \eqref{eq3}, we obtain
\begin{eqnarray*}
\int_{T_1}^T(Y_{s}-S_{s})dK_{s}=0.
\end{eqnarray*}
Finally, we can deduce that the process $(Y,Z,K)$ is solution of RBDSDE \eqref{1.6} starting at $T_1$ with horizon $T$.
If $T_1=0$, the proof of existence is finished. But if $T_1\neq 0$, we need to prove an existence result the following equation:
\begin{eqnarray}  \label{2e}
\left\{
\begin{array}{ll}
&\displaystyle (i)\,Y_t=\xi+\int_{t}^{T_1}f(s,Y_s,Z_s)ds
\displaystyle+\int_{t}^{T_1}g(s,Y_s,Z_s)\overleftarrow{dB}_{s}+\int_{t}^{T_1}dK_{s}-\int_{t}^{T_1}Z_sdW_{s},\; \; t\in[0,T_1],\\\\
&\displaystyle (ii)\,Y_{t}\geq S_{t},\; \; t\in[0,T_1],\\\\
&\displaystyle (iii)\,\int_{0}^{T_1}(Y_{t}-S_{t})dK_{t}=0  \hbox{.}
\end{array}
\right.
\end{eqnarray}
Repeating  the setup as above, we set for all $t\in [0,T_1],\; Y^{0}(t)=Z^{0}(t)=0$ and for all $n\in \N$, we define recursively the reflected BDSDEs
\begin{eqnarray}  \label{eq3bis}
\left\{
\begin{array}{ll}
&\displaystyle (i)\,Y_t^n=\xi+\int_{t}^{T_1}f(s,Y_s^{n-1},Z_s^{n})ds
\displaystyle+\int_{t}^{T_1}g(s,Y_s^{n-1},Z_s^n)\overleftarrow{dB}_{s}+\int_{t}^{T_1}dK_{s}^{n}-\int_{t}^{T_1}Z_s^ndW_{s},\\\\
&\displaystyle (ii)\,Y_{t}^n\geq S_{t},\\\\
&\displaystyle (iii)\,\int_0^{T_1}(Y_{t}^n-S_{t})dK_{t}^n=0  \hbox{.}
\end{array}
\right.
\end{eqnarray}
The same procedure used  in the proof of Lemmas \ref{l2} and Lemma \ref{l3}, leads for all $t\in[T_1,T]$ and  $n,\ m \geq 1$,
\begin{eqnarray*}
{\E}\left| Y_{t}^{n+m}-Y_{t}^{n}\right| ^{2} \leq
e^{\frac{CT}{1-\alpha}}\left(\frac{1-\alpha}{C}+1\right)\int_{t}^{T_1
}\rho(s,{\E}\left| Y_{s}^{n+m-1}-Y_{s}^{n-1}\right| ^{2})ds,
\end{eqnarray*}
and
\begin{eqnarray*}\label{}
{\E}\left|Y_{t}^{n}\right|^{2} &\leq&\mu_t^2+ce^{cT}\int_{t}^{T_1 }
\rho(s,{\E}\left|Y_{s}^{n-1}\right|^2) ds
\end{eqnarray*}
where
$$\mu_t^2=ce^{cT}\left(1+{\E}\left|Y_{T_1}\right|
^{2}+{\E}\underset{0\leq t\leq
T}{\sup}|S_s|^2+{\E}\int_{t}^{T
}\left[\left|f(s,0,0)\right|^2\!+\!
\left|g(s,0,0)\right|^2\right]ds\right).$$
Letting $M_2=2\mu_2^2$, we can also find $T_2\in[0,T_1[$ such that $\int_{T_2}^{T_1 }\rho(s,M_2) ds= \frac{\mu_0^2}{M}$ and
\begin{eqnarray*}
{\E}\left|Y_{t}^{n}\right| ^{2} &\leq & M_2,\ \ n\geq 1, \ t\in
[T_2,T_1].
\end{eqnarray*}
As above, we prove the existence of solution of RBDSDE \eqref{1.6} starting at $T_2$ with horizon $T_1$. If
$T_2=0$, the proof of the existence is complete. Otherwise, we
repeat the above processes. Thus, we obtain a sequence $\{T_p,\
\mu_t^p,\ M_p,\ \ p\geq 1\}$ defined by
\begin{eqnarray*}
&&0\leq T_p< T_{p-1}<...<T_1<T_0=T,\\
&&\mu_t^p=ce^{cT}\left(1+{\E}\left|Y_{T_{p-1}}\right|
^{2}+{\E}\underset{0\leq t\leq
T}{\sup}|S_s|^2+{\E}\int_{t}^{T
}\left[\left|f(s,0,0)\right|^2\!+\!
\left|g(s,0,0)\right|^2\right]ds\right),\\
&&M_p=2\mu_0^p \ \text{\ \ and \ \ } \int_{T_{p}}^{T_{p-1} } \rho(s,M_p) ds=
\frac{\mu_0^p}{M}.
\end{eqnarray*}
Therefore, by iteration, we construct a solution of RBDSDE \eqref{1.6} starting at $0$ with horizon $T$.
Finally, by the same argument used in \cite{MNJMO2}, there exists
a finite $p\geq 1$ such that $T_p=0$. Thus, we obtain
the existence of the solution of RBDSDEs \eqref{1.6} on  $[0,T]$.

{\bf  Uniqueness.}\\
Let $\left(Y,Z,K\right)$ and $\left(Y',Z',K'\right)$ belong in $\mathcal{E}^2([0,T])$ be two
solutions of the RBDSDE \eqref{1.6}. By virtue of
It\^o's formula, we have for any $\theta>0$
\begin{eqnarray*}\label{}
&&{\E}|Y_t-Y'_t|^2e^{\theta
t}+\theta{\E}\int_{t}^{T}|Y_s-Y'_s|^2e^{\theta
s}ds+{\E}\int_{t}^{T}|Z_s-Z'_s|^2e^{\theta
s}ds\\&&=2{\E}\int_{t}^{T}\left( Y_s-Y'_s\right)
\left(f(s,Y_s,Z_s)-f(s,Y'_s,Z'_s)\right) e^{\theta s}ds+
2{\E}\int_{t}^{T}\left( Y_s-Y'_s\right) e^{\theta s}(dK_s-dK'_s)\\&&+
{\E}\int_{t}^{T}|g(s,Y_s,Z_s)-g(s,Y'_s,Z'_s)|^2e^{\theta s}ds.
\end{eqnarray*}
Since $\displaystyle\int_{t}^{T}\left( Y_s-Y'_s\right) e^{\theta s}(dK_s-dK'_s)\leq0$, it follows from $({\bf H1})$, $({\bf H4})$ and Young's inequality
$2ab\leq \frac{1}{\theta}a^2+\theta b^2 $ that
\begin{eqnarray*}\label{}
&&{\E}|Y_t-Y'_t|^2e^{\theta
t}+(1-\alpha-\frac{1}{\theta} C){\E}\int_{t}^{T}|Z_s-Z'_s|^2e^{\theta
s}ds\\&&\leq
\left(\frac{1}{\theta}+1\right){\E}\int_{t}^{T}\rho(s,|Y_s-Y'_s|^2)e^{\theta
s}ds.
\end{eqnarray*}
Choosing $\theta=\frac{2C}{1-\alpha}$, 
 we get for all $ t\in[0,T]$,
\begin{eqnarray}\label{ST}
&&{\E}|Y_t-Y'_t|^2+\frac{1-\alpha }{2}{\E}\int_{t}^{T}|Z_s-Z'_s|^2ds\\&&\leq
e^{\frac{2C T}{1-\alpha}
}\left(\frac{1-\alpha}{2C}+1\right){\E}\int_{t}^{T}\rho(s,|Y_s-Y'_s|^2)ds \notag.
\end{eqnarray}
Therefore
\begin{eqnarray*}\label{}
{\E}|Y_t-Y'_t|^2\leq e^{\frac{2C T}{1-\alpha}
}\left(\frac{1-\alpha}{2C}+1\right)\int_{t}^{T}\rho(s,{\E}|Y_s-Y'_s|^2)ds.
\end{eqnarray*}
In view of the comparison Theorem for ODE, we have
\begin{eqnarray*}\label{}
{\E}|Y_t-Y'_t|^2\leq r(t), \ \ \forall\ t\in[0,T],
\end{eqnarray*}
where $r(t)$ is the maximum left shift solution of the following
equation:
\begin{eqnarray*}
\left\{
\begin{array}{ccc}
  u' & = & -e^{\frac{2C T}{1-\alpha}
}\left(\frac{1-\alpha}{2C}+1\right)\rho(t,u); \\
  u(T) & = & 0. \text{ \ \ \ \ \ \ \ \ \ \ \ \ \ \ \ \ \ \ \ \ \ \ \ \ \ \  } \\
\end{array}
\right.	
\end{eqnarray*}
By virtue of (H3), $r(t)=0$,
$t\in[0,T]$. Hence, $Y_t=Y'_t$, a.s., for any $t\in[0,T]$. It then follows from \eqref{ST} that
$Z_t=Z'_t$, a.s., for any $t\in[0,T]$. On the other hand, since
\begin{eqnarray*}
K_t-K'_t&=&Y_0-Y'_0-(Y_t-Y'_t)-
\int_{0}^{t}(f(s,Y_s,Z_s)-f(s,Y'_s,Z'_s))ds\\
&&-\int_{0}^{t}(g(s,Y_s,Z_s)-g(s,Y'_s,Z'_s))\overleftarrow{dB}_{s}+\int_{0}^{t}(Z_s-Z'_s)dW_{s},\ \ t \in [0,T],
\end{eqnarray*} we have, $K_t=K'_t$, a.s., for any $ t \in [0,T]$ which end the proof of the Theorem.
\end{proof}
\subsection{Comparison principle for reflected Backward doubly SDE}
Let $(\xi^1,f^1,S^1)$ and $(\xi^2,f^2,S^2)$ be two set of data, each one satisfying the conditions of Theorem \ref{TE}, and suppose in additional the following
\begin{description}
\item [(H6)]
\begin{itemize}
\item [$(i)$] $\xi^1\leq \xi^2$, a.s., 
\item [$(ii)$]$f^1(t,Y^1_t,Z^1_t)\leq f^2(t,Y^1_t,Z^1_t)$ or $f^1(t,Y^2_t,Z^2_t)\leq f^2(t,Y^2_t,Z^2_t)$, a.s., for a.e. $t\in[0,T]$,
\item [$(iii)$]  $S^1_t\leq S^2_t$, a.s., for a.e. $t\in[0,T]$, 
\end{itemize} 
\end{description}
where $(Y^1,Z^1,K^1)$ is a solution of RBSDE with data $(\xi^1,f^1,S^1)$ and $(Y^2,Z^2,K^2)$ is a solution of RBSDE with data $(\xi^2,f^2,S^2)$. Then we have the following comparison theorem.
\begin{theorem}
Assume the conditions of Theorem \ref{TE} and $({\bf H6})$ hold. Then $Y^1_t\leq Y^2_t$ a.s. $\forall\; t\in[0,T]$. 
\end{theorem}
\begin{proof}
We shall assume that $f^1(t,Y^1_t,Z^1_t)\leq f^2(t,Y^1_t,Z^1_t)$, a.s., a.e. $t\in[0,T]$, and denote $\overline{Y}_t=Y^1_t-Y^2_t,\; \overline{Z}_t=Z^1_t-Z^2_t$ and $\overline{K}=K^1-K^2$.
Applying Itô formula to $|\overline{Y}^+_t|^2$, and taking the expectation, we have
\begin{eqnarray*}\label{C2}
&&\E\left[|\overline{Y}^+_t|^2+\int^T_t{\bf 1}_{\{\overline{Y}>0\}}|\overline{Z}_s|^2ds\right]\nonumber\\
&=&\E\left[|(\xi^1-\xi^2)^{+}|^2+2\int^T_t\overline{Y}^+_s(f^1(s,Y^1_s,Z^1_s)-f^2(s,Y^2_s,Z^2_s))ds\right.\\
&&+\left.2\int^T_t\overline{Y}^+_s d\overline{K}_s+\int^T_t{\bf 1}_{\{\overline{Y}_s>0\}}|g(s,Y^1_s,Z^1_s)-g(s,Y^2_s,Z^2_s)|^2ds\right].
\end{eqnarray*}
Since on $\{Y^1_t>Y^2_t\}$, $Y^1_t> Y^2_t\geq S^2_t\geq S^1_t$, we have
\begin{eqnarray*}
\int^T_t\overline{Y}^+_sd\overline{K}_s\leq -\int^T_t\overline{Y}^+_sdK^2_s\\
&\leq & 0.
\end{eqnarray*}

Assume now that $({\bf H6})$ in the statement applies to $f$ and $g$. Then
\begin{eqnarray}\label{C2}
&&\E\left[|\overline{Y}^+_t|^2+\int^T_t{\bf 1}_{\{\overline{Y}>0\}}|\overline{Z}_s|^2ds\right]\nonumber\\
& \leq &\left[ \int^T_t\overline{Y}^+_s(f^1(s,Y^1_s,Z^1_s)-f^1(s,Y^2_s,Z^2_s))ds+\int^T_t{\bf 1}_{\{\overline{Y}_s>0\}}|g(s,Y^2_s,Z^2_s)-g(s,Y^1_s,Z^1_s)|^2ds\right].\nonumber\\
\end{eqnarray}
On the other hand, using $({\bf H3})$ and the basic inequality $2ab\leq \delta a^2+\frac{1}{\delta}b^2$ we get
 \begin{eqnarray}\label{C6}
 \int^T_te^{\mu s}\overline{Y}^+_s(f^1(s,Y^1_s,Z^1_s)-f^1(s,Y^2_s,Z^2_s))ds &\leq & \delta\int_t^Te^{\mu s}|\overline{Y}^+_s|^2ds+\frac{1}{\delta}\int_t^Te^{\mu s}\rho(s,|\overline{Y}_s|^2){\bf 1}_{\{\overline{Y_s}>0\}}ds\nonumber\\
 &&+\frac{c}{\delta}\int_t^Te^{\beta s}|\overline{Z}_s|^2{\bf 1}_{\{\overline{Y_s}>0\}}ds\nonumber\\
 &\leq & \delta\int_t^Te^{\mu s}|\overline{Y}^+_s|^2ds+\frac{1}{\delta}\int_t^Te^{\mu s}\rho(s,|\overline{Y}^+_s|^2)ds\nonumber\\
 &&+\frac{c}{\delta}\int_t^Te^{\beta s}|\overline{Z}_s|^2ds,
 \end{eqnarray}
 and 
 \begin{eqnarray}\label{C8}
 &&\int^T_t{\bf 1}_{\{\overline{Y}_s>0\}}e^{\mu s}|g(s,Y^2_s,Z^2_s)-g(s,Y^1_s,Z^1_s)|^2ds\nonumber\\&\leq & \int_t^Te^{\mu s}\rho(s,|\overline{Y}_s|^2){\bf 1}_{\{\overline{Y_s}>0\}}ds+\alpha\int_t^Te^{\mu s}|\overline{Z}_s|^2{\bf 1}_{\{\overline{Y_s}<0\}}ds\nonumber\\
 &\leq & \int_t^Te^{\mu s}\rho(s,|\overline{Y}^{+}_s|^2)ds+\alpha\int_t^Te^{\mu s}|\overline{Z}_s|^2ds.
 \end{eqnarray}
Putting \eqref{C6}-\eqref{C8} in \eqref{C2} and since $ 0<\alpha<1$, we get
\begin{eqnarray*}
&&\E(e^{\mu t}|\overline{Y}_t^+|^2)+(\mu-\delta)\E\left(\int_t^Te^{\mu s}|\overline{Y}^+_s|^2ds\right)+\left(1-\alpha-\frac{\alpha}{\delta}\right)\E\left(\int_t^Te^{\mu s}|\overline{Y}^+_s|^2ds\right)\nonumber\\
&\leq &
\left(\frac{1}{\delta}+1\right)\E\left(\int^T_te^{\mu s}\rho(s,|\overline{Y}_s^+|^2)ds\right).
\end{eqnarray*}
Finally choosing $\mu>0$ and $\delta>0$ such that $\mu-\delta>0$ and $ 1-\alpha-\frac{\alpha}{\delta}>0$, we have 
\begin{eqnarray*}
\E(e^{\mu t}|\overline{Y}_t^+|^2)&\leq &
C\E\left(\int^T_te^{\mu s}\rho(s,|\overline{Y}_s^+|^2)ds\right),
\end{eqnarray*}
which by using Fubini's theorem and Jensen's inequality leads to
\begin{eqnarray*}
\E(|\overline{Y}_t^+|^2)&\leq &
C\int^T_t\rho(s,\E(|\overline{Y}_s^+|^2))ds.
\end{eqnarray*}
Using the same argument as in the proof of uniqueness, we get that
$\E(|\overline{Y}_t^+|^2)=0$ and hence $\overline{Y}_t^+=0$ which implies that $Y^1_t\leq Y^2_t$.
\end{proof}

\section{Obstacle problem for a non linear parabolic stochastic PDEs with non Lipschitz coefficient}

The goal of this section, is to derive existence of stochastic viscosity solution to class of reflected stochastic PDEs called "obstacle problem" for SPDE. Roughly speaking, SPDEs is of the form
Let us consider the following related obstacle problem for a parabolic SPDE
\begin{eqnarray}\label{5edps}
SPDE^{(f,g,h,l)}\left\{
\begin{array}{ll}
&\displaystyle \min\Big\{u(t,x)- h(t,x)\;,\;\frac{\partial u}{\partial t}(t,x)+Lu(t,x)+f(t,x,u(t,x),(\sigma^*Du)(t,x))\\\\
&\displaystyle\hspace{3.5cm}+g(t,x,u(t,x))\overleftarrow{\dot{B}_{t}}\Big\},\; \; \; \; (t,x)\in[0,T]\times \mathbb{R}^q, \\\\
&\displaystyle u(T,x)= l(x)\hbox{,}\; \; \; \;x\in\mathbb{R}^q.
\end{array}
\right.
\end{eqnarray}
where, $\dot{B}_{t}=dB_{t}/dt$ and $L$ is a second order differential operator defined by
\begin{eqnarray*}
L=\frac{1}{2}\sum_{i,j=1}^d(\sigma\sigma^*)_{ij}\frac{\partial^2}{\partial x_i\partial x_j}+\sum_{i=1}^d b_i\frac{\partial}{\partial x_i}.	
\end{eqnarray*}
Our method is fully probabilistic and used reflected BDSDEs studied in the previous sections and is done when data $l: \R^{d}\rightarrow\mathbb{R}$,\ \ $f:\Omega_2 \times [0,T]\times
\R^{q}\times \mathbb{R}^{} \times \mathbb{R}^{ d}\rightarrow
\mathbb{R}^{},$ \ \ $g:\Omega_2 \times [0,T]\times \R^{d}\times \mathbb{R}^{}
\rightarrow \mathbb{R}^{l}$ and $h: \Omega_2 \times[0,T]\times\R^{d}\rightarrow \mathbb{R}^{}$ satisfy the following Assumptions.
\begin{description}
%\begin{itemize}
  \item[(\bf {A1})] $l$ is Lipschitz continuous with the common Lipschitz constant $C$,

  \item[(\bf {A2})] $h$ is continuous such that, $|h(t,x)|\leq C(1+|x|^p),\ (t,x)\in [0,T]\times\R^{d}$ and $h(T,x)\leq l(x),$

\item[(\bf {A3})] for any  $ (\omega_2, y, z) \in \Omega_2 \times
\mathbb{R} \times \mathbb{R}^{d}$ and $t\in[0,T]$, $x_{1},x_{2} \in
\times \mathbb{R}^{q}$,
$$|f(t,x_{1},y,z)-f(t,x_{2},y,z)| \leq
 C| x_{1}-x_{2}|,$$
\item[{\bf(A4)}]$g\in C_b^{0,2,3}([0,T]\times \R^{q}\times \R; \R^l)$ 

\item[{\bf(A5)}]for all $(y_{1},z_{1}),\ (y_{2},z_{2})\in \mathbb{R}%
 \times \mathbb{R}^{d}$, $t\in [0,T]$ and $x\in\R^q$,
\begin{eqnarray*}
\left\{
\begin{array}{ll}
|f(t,x,y_{1},z_{1})-f(t,x,y_{2},z_{2}) |^{2} \leq \rho(t,|
y_{1}-y_{2}|^{2})+C| z_{1}-z_{2}|^{2} \vspace{0.2cm} &
\\\\
|f(t,0,y,0)| \leq \varphi(t)+C|y|,\; \mbox{with}\;  \varphi(.) \in\mathcal{M}^{2}(\textbf{\textit{F}},[0,T])&
\end{array}
\right.,
\end{eqnarray*}
where $C > 0$ is a constants and $\rho: [0,T]\times
\mathbb{R}^+ \rightarrow \mathbb{R}^+$ satisfies:

\begin{itemize}
\item[(i)] for fixed $t\in[0,T]$, \ $\rho(t,.)$ is concave and
non-decreasing such that $\rho(t,0)=0.$

\item[(ii)] for fixed $u$, $\int_0^T\rho(t,u)dt<+\infty$
\item[(iii)] for any $M>0$, the following ODE
\begin{equation*}
\left\{
\begin{array}{ccc}
u^{\prime } & = & -M\rho (t,u) \\
u(T) & = & 0%
\end{array}%
\right.
\end{equation*}%
has a unique solution $u(t)\equiv 0,\ \ t\in \lbrack 0,T]$,
\end{itemize}
\end{description}
In addition, we will consider the following. For each $t\geq 0$, we consider ${\bf F}^t=\{\mathcal{F}^t_s\}_{t\leq s\leq T}$ defined by $\mathcal{F}^t_s=\mathcal{F}_{t,s}^{W}\otimes\mathcal{F}^{B}_{s,T}$ and $\mathcal{M}_{0,T}^{B}$ denote the set of all $\textbf{\textit{F}}_{}^{B}$-stopping times $\tau$ such that $0\leq \tau\leq T,$ $\P_{2}-$almost surely. For generic Euclidean spaces $E$ and $E_1$, we introduce the following vector spaces of functions:

\medskip\noindent$\bullet$ $\mathcal{C}^{k,n}([0, T] \times E;E_1)$  design the space of all functions defined on $[0, T] \times E$ with values in $E_1$, which are $k-$ times continuously differentiable in $t$ and $n-$ times continuously differentiable in $x$ and $\mathcal{C}_b^{k,n}([0, T] \times E;E_1)$ denotes the subspace of $\mathcal{C}^{k,n}([0, T] \times E;E_1)$ which contains all uniformly bounded partial derivatives functions;

%\item $
\medskip\noindent$\bullet$ For any sub $\sigma-$field $\mathcal{G}\subset\mathcal{F}_{T}^{B}$,\
$\mathcal{C}^{k,n}(\mathcal{G}, [0, T] \times E;E_1)$, (resp.  $\mathcal{C}_b^{k,n}(\mathcal{G}, [0, T] \times E;E_1)$ stands for the space of all random variables with values in $\mathcal{C}^{k,n}( [0, T] \times E;E_1)$, (resp.  $\mathcal{C}_b^{k,n}( [0, T] \times E;E_1)$ which are $\mathcal{G}\otimes\mathcal{B}([0,T]\times E)-$measurable;

%\item $
\medskip\noindent$\bullet$ $\mathcal{C}^{k,n}(\textbf{\textit{F}}_{}^{B}, [0, T] \times E;E_1)$, (resp.
$\mathcal{C}_b^{k,n}(\textbf{\textit{F}}_{}^{B}, [0, T] \times E;E_1)$ is the space of random fields $\varphi\in \mathcal{C}^{k,n}(\mathcal{F}_T^{B}, [0, T] \times E;E_1)$, (resp.  $\mathcal{C}_b^{k,n}(\mathcal{F}_T^{B}, [0, T] \times E;E_1)$ such that for any $x\in E$, the mapping $(\omega^2,t)\mapsto \varphi(\omega_2,t,x)$ is $\textbf{\textit{F}}_{}^{B}-$progressively measurable.

%\item $
\medskip\noindent$\bullet$  For any sub $\sigma-$field $\mathcal{G}\subset\mathcal{F}_{T}^{B}$,\ $\mathcal{LSC}([0, T] \times E;E_1)$ (resp. $\mathcal{USC}([0, T] \times E;E_1)$) designs the space of all lower (resp. upper) semi continuous functions defined on $[0, T] \times E$ with values in $E_1$;

%\item $
\medskip\noindent$\bullet$ $\mathcal{LSC}(\mathcal{G}, [0, T] \times E;E_1)$, (resp.  $\mathcal{USC}(\mathcal{G}, [0, T] \times E;E_1)$ stands for all random variables with values in $\mathcal{LSC}( [0, T] \times E;E_1)$, (resp.  $\mathcal{USC}( [0, T] \times E;E_1)$ which are $\mathcal{G}\otimes\mathcal{B}([0,T]\times E)-$measurable;

%\item $
\medskip\noindent$\bullet$ $\mathcal{LSC}(\textbf{\textit{F}}_{}^{B}, [0, T] \times E;E_1)$, (resp.
$\mathcal{USC}(\textbf{\textit{F}}_{}^{B}, [0, T] \times E;E_1)$ denotes the space of random fields $\varphi\in \mathcal{LSC}(\mathcal{F}_T^{B}, [0, T] \times E;E_1)$, (resp.  $\mathcal{USC}(\mathcal{F}_T^{B}, [0, T] \times E;E_1)$ such that for any $x\in E$, the mapping $(\omega_2,t)\mapsto \varphi(\omega_2,t,x)$ is $\textbf{\textit{F}}_{}^{B}-$progressively measurable.

%\item $
\medskip\noindent$\bullet$ For any sub $\sigma-$field $\mathcal{G}\subset\mathcal{F}_{T}^{B}$,\ and for any $p\geq 0$,
$L^{p}(\mathcal{G}, E)$ design the space of $\mathcal{G}-$measurable random variables $\xi$ with values in $E$ such that $\E |\xi|^p<\infty.$
%\end{itemize}

\medskip \noindent Furthermore, for any $(t,x,y)\in[0,T]\times\R^{q}\times\R$, we denote $D=D_x=(\frac{\partial}{\partial x_1},\ldots, \frac{\partial}{\partial x_q}),$ $D_y=\frac{\partial}{\partial y},$\  $D_t=\frac{\partial}{\partial t},$\ and\ $D_{xx}=\left(\partial^2_{ x_ix_j}\right)_{i,j=1,\dots,q}.$ The meaning of $D_{xy}$ and $D_{yy}$ is then self-explanatory.

\medskip\noindent We note that
$$
\mathcal{C}^{0,0}(\textbf{\textit{F}}_{}^{B}, [0, T] \times E;E_1)=\mathcal{LSC}(\textbf{\textit{F}}_{}^{B}, [0, T] \times E;E_1)\cap \mathcal{USC}(\textbf{\textit{F}}_{}^{B}, [0, T] \times E;E_1).
$$
\subsection{Notion of stochastic viscosity solution}
A solution of the obstacle problem for SPDEs $(f, g, h, l)$ is a random field $u:\Omega_2\times[0,T]\times \mathbb{R}^q\rightarrow \mathbb{R}$ which satisfies \eqref{5edps}.
More precisely, in this section, we will consider the solution of SPDE \eqref{5edps} with data $(f, g, h, l)$ in the
stochastic viscosity sense, inspired by the works of Buckdahn and Ma \cite{5, 6} and B. Djehiche, N'zi and Owo \cite{DON}. To this end, we define the process $\eta \in \mathcal{C}^{0,0,0}([0, T] \times \R^{q}\times \R; \R)$ as the solution to the following SDE,
\begin{eqnarray*}\label{}
\eta(t,x,y)&=&y+\frac{1}{2}\int_{t}^{T}\langle g,D_yg\rangle(s,x,\eta(s,x,y))ds\notag\\&&
+\int_{t}^{T}\langle g(s,x,\eta(s,x,y)),\overleftarrow{dB_s}\rangle,\ \ \ 0\leq
t\leq T.
\end{eqnarray*}
Under condition $({\bf A4})$, the mapping $y\mapsto \eta(t,x,y)$ is a diffeomorphism for all $(t,x)$, $\P^{2}-a.s.$ such that
$ \eta\in\mathcal{C}^{0,2,2}(\textbf{\textit{F}}_{}^{B},[0, T] \times \R^{q}\times \R; \R)$. Let $\varepsilon(t,x,y)$ denotes the $y-$inverse of $\eta(t,x,y)$. Then since $\varepsilon(t,x,\eta(t,x,y))=y$ one can show that (see Buckdahn and Ma \cite{5, 6})
\begin{eqnarray}\label{5edst}
\varepsilon(t,x,y)&=&y-
\int_{t}^{T}\langle D_y\varepsilon(s,x,y) ,g(s,x,y)\circ \overleftarrow{dB_s}\rangle,\ \ \ 0\leq
t\leq T.\notag
\end{eqnarray}
Furthermore, if
$\psi(t,x)=\eta(t,x,\varphi(t,x)),$ for $(t,x)\in[0, T] \times \R^{q},$ then $\psi\in \mathcal{C}^{0,p}(\textbf{\textit{F}}_{}^{B},[0, T] \times \R^{q}; \R)$ if and only if $\varphi\in \mathcal{C}^{0,p}(\textbf{\textit{F}}_{}^{B},[0, T] \times \R^{q}; \R)$, for $p=0,1,2$. Next in order to simplify the notation, we set
\begin{eqnarray*}
\mathcal{A}_{f,g}(\psi(t,x))=-L\psi(t,x)-
f\big(t,x,\psi(t,x),\sigma^*(t,x)D_x\psi(t,x)\big)+\frac{1}{2}\langle g,D_yg\rangle(t,x,\psi(t,x)).	
\end{eqnarray*}
We now give the definition of stochastic viscosity solution of the reflected
SPDE$(f, g, h, l)$.

\begin{definition}\label{5df1}
$(a)$ A random field $u\in \mathcal{LSC}(\textbf{\textit{F}}_{}^{B}, [0, T] \times \R^{q};\R)$ is said to be a stochastic viscosity subsolution of SPDE$(f,g,h,l)$ if $u(T,x)\leq l(x),$ for all $\ x\in\R^{q};$ and if for any stopping time $\tau\in\mathcal{M}_{0,T}^{B},$ any state variable $ \xi\in L^{0}(\mathcal{F}_{\tau}^{B}, \R^{q}) $, and any random field $\varphi\in\mathcal{ C}^{1,2}(\mathcal{F}_{\tau}^{B}, [0, T] \times \R^{q};\R)$ such that, for $\P_2-$almost all \ $\omega_{2}\in\left\{0<\tau<T\right\}$, it holds
$$
u(\omega_{2}, t,x)-\psi(\omega_{2}, t,x)\leq 0=u(\tau(\omega_{2}),\xi(\omega_{2}))-\psi(\tau(\omega_{2}),\xi(\omega_{2})),
$$
for all $(t,x)$ in a neighborhood of $(\tau(\omega_{2}),\xi(\omega_{2}))$,  where $\psi(t,x)\overset{\Delta}{=}\eta(t,x,\varphi(t,x))$, then we have, $\P_{2}-a.s.$ on $\left\{0<\tau<T\right\},$
\begin{eqnarray}\label{df1}
\min\Big(u(\tau,\xi)-h(\tau,\xi)\;,\;\mathcal{A}_{f,g}(\psi(\tau,\xi))- D_y\psi(\tau,\xi)D_t\varphi(\tau,\xi)\Big)\leq 0,
\end{eqnarray}

\medskip\noindent
$(b)$ A random field $u\in \mathcal{USC}(\textbf{\textit{F}}_{}^{B}, [0, T] \times \R^{q};\R)$ is said to be a stochastic viscosity supersolution of SPDE$(f,g,h,l)$ if $u(T,x)\geq l(x),$ for all $\ x\in\R^{q};$ and if for any stopping time $\tau\in\mathcal{M}_{0,T}^{B},$ any state variable $ \xi\in L^{0}(\mathcal{F}_{\tau}^{B}, \R^{q}) $, and any random field $\varphi\in\mathcal{ C}^{1,2}(\mathcal{F}_{\tau}^{B}, [0, T] \times \R^{q};\R)$ such that, for $\P_2-$almost all \ $\omega_{2}\in\left\{0<\tau<T\right\}$, it holds
$$
u(\omega_{2}, t,x)-\psi(\omega_{2}, t,x)\geq 0=u(\tau(\omega_{2}),\xi(\omega_{2}))-\psi(\tau(\omega_{2}),\xi(\omega_{2})),$$ for all $(t,x)$ in a neighborhood of $(\tau(\omega_{2}),\xi(\omega_{2}))$, then we have, $\P_{2}-a.s.$ on $\left\{0<\tau<T\right\},$
\begin{eqnarray}\label{df2}
\min\Big(u(\tau,\xi)-h(\tau,\xi)\;,\;\mathcal{A}_{f,g}(\psi(\tau,\xi))- D_y\psi(\tau,\xi)D_t\varphi(\tau,\xi)\Big)\geq 0,
\end{eqnarray}

\medskip\noindent
$(c)$ A random field $u$ is said to be a stochastic viscosity solution of SPDE$(f,g,h,l)$ if $u\in \mathcal{C}^{0,0}(\textbf{\textit{F}}_{}^{B}, [0, T] \times \R^{q};\R)$ and is both a stochastic viscosity subsolution and supersolution.
\end{definition}

\begin{remark}
If, in SPDE$(f,g,h,l)$, $g\equiv 0$, then for all $(t,x,y)$, $\eta(t,x,y)=y$ and  $\psi(t,x)=\varphi(t,x)$. Hence, if $f$ is deterministic, the above definition coincides with the deterministic case (El Karoui et al. \cite{15}). Thus, any stochastic viscosity (sub- or super-) solution is viewed as a (deterministic) viscosity (sub- or super-) solution for each fixed $\omega_{2}\in\left\{0<\tau<T\right\}$, modulo the $\mathcal{F}_{\tau}^{B}-$ measurability requirement of the test function $\varphi$.
\end{remark}
\subsection{Existence of stochastic viscosity solution}
\label{sect2}
This subsection is devoted to prove the existence of stochastic viscosity solutions to obstacle problem for SPDE \eqref{5edps} using the result of Section 2. Before giving the main result, let state the Markovian framework of decoupled forward-backward
SDE. For $b:\R^{d}\rightarrow \R^{d}, \, \sigma; \R^{d}\rightarrow \R^{d\times d}$ are uniformly Lipschitz continuous (with a common Lipschitz constant $C>0$), let consider this needed progressive SDE and the following regularity result associated to it (see the theory of SDEs, for more detail): for each $(t,x)\in [0,T]\times\R^{d}$,
\begin{eqnarray}\label{5eds}
X_s^{t,x} = x+\int_{t}^{s} b(X_r^{t,x})dr+\int_{t}^{s}
\sigma(X_r^{t,x}) dW_r, \ \ \ s\in[t,T],
\end{eqnarray}

\begin{proposition}\label{pro} There exists a constant $C > 0$ such that for all $t,t'\in [0,T]$ and $x,x'\in\R^q$.
\begin{eqnarray}\label{}
\E\left(\sup_{0\leq s\leq T}|X_s^{t,x}-X_s^{t',x'}|^{p}\right) \leq C(|t-t'|^{p|2}+ |x-x'|^{p}).
\end{eqnarray}
\end{proposition}
Next, let us consider RBDSDE$(l(X_{T}^{t,x}),f, g, h)$:
\begin{eqnarray}\label{5eddsr}
\left\{
\begin{array}{ll}
&\displaystyle (i)\; Y_{s}^{t,x}=l(X_T^{t,x})
+\int_{s}^{T}f(r,X_r^{t,x},Y_{r}^{t,x},Z_{r}^{t,x})dr
\displaystyle+\int_{s}^{T}g(r,X_r^{t,x},Y_{r}^{t,x})\overleftarrow{dB_r}\\
&\displaystyle\; \; \; \hspace{1cm}+K_{T}^{t,x}-K_{s}^{t,x}-\int_{s}^{T}Z_{r}^{t,x}dW_{r},\; \; \; \; \; s\in[t,T],\\\\
&\displaystyle (ii)\; Y_{s}^{t,x}\geq h(s,X_s^{t,x}),\; \; \; \; \; s\in[t,T],\\\\
&\displaystyle (iii)\; \{K_{s}^{t,x}\}\ \text{is increasing and continuous such that}\ K_{0}^{t,x}=0\\& \hspace{.5cm}\text{and}\  \displaystyle\int_t^T(Y_{r}^{t,x}-h(r,X_r^{t,x}))dK_{r}^{t,x}=0  \hbox{.}
\end{array}
\right.
\end{eqnarray}
According to Theorem \ref{TE} of sections 2, for each $(t,x)\in[0,T]\times\R^d$, RBDSDE \eqref{5eddsr} has a unique solution $(Y^{t,x},Z^{t,x},K^{t,x})\in\mathcal{E}^{2}([t,T])$. We can extend this solution to $[0,t]$ by choosing $Y_{s}^{t,x}=Y_{t}^{t,x},\, Z_{s}^{t,x}=0,\, K_{t}^{t,x}=K_{t}^{t,x}$. Furthermore, we have
\begin{proposition}
Let $u:\Omega_2\times[0,T]\times \mathbb{R}^d\rightarrow \mathbb{R}$ be a random field defined by 
\begin{eqnarray}\label{ao}
u(t,x)\overset{\Delta}{=}Y_{t}^{t,x}, \ \ \text{for all} \ \ (t,x)\in [0,T]\times \mathbb{R}^q. \end{eqnarray}
Then, $u\in \mathcal{C}^{0,0}(\textbf{\textit{F}}_{}^{B},[0, T] \times \R^{d}; \R)$.
\end{proposition}

\begin{proof}
For $n\in\N$, let
\begin{eqnarray*}
\underline{f}_{n}(t,x,y,z)=\underset{u \in \Q}\inf \left\{f(t,x,u,z)+n|y-u|\right\}.
\end{eqnarray*}
and
\begin{eqnarray*}
\overline{f}_{n}(t,x,y,z)=\underset{u \in \Q}\sup \left\{f(t,x,u,z)-n|y-u|\right\}.
\end{eqnarray*}
Since $f$ is continuous, with linear growth (see assumption $({\bf A5})$, it follows from Lepeltier and San Martin \cite{LS} or
K. Bahlali et al. \cite{4} that, for all $n\geq C$ and $( t, x, y, z),( t_{i}, x_{i}, y_{i}, z_{i})\in [0,T]\times \mathbb{R}^d\times \mathbb{R}\times \mathbb{R}^d$, $i=1,2$,
\begin{itemize}
\item[(i)] \ $\underline{f}_{n}( t, x, y, z) \leq f( t, x, y, z) \leq \bar{f}_{n}( t, x, y, z)$;
\item[(ii)] \ $\underline{f}_n( t, x, y, z)$ is non-decreasing in $n$ and $\overline{f}_n( t, x, y, z)$ is non-increasing in $n$.
\end{itemize}
Taking $\phi$ as $\overline{f}_{n}$ or $\underline{f}_{n}$
\begin{itemize}
\item[ (iii)] \ $ |\phi( t, x, y, z) | \leq \varphi_t+C(|x|+| y |+|z|)$.
\item[(iv)] \ $| \phi(t,x_{1},y_1,z)-\phi(t,x_{2},y_2,z)| \leq
n(| x_{1}-x_{2}|+ | y_{1}-y_{2}|)$. 
\item[(v)] \ $| \phi(t,x,y,z_1)-\phi(t,x,y,z_2)|^2 \leq
C| z_1-z_2|^2$.
\item[ (vi)] If  $(y_n,z_n) \rightarrow (y,z)$, then
$\phi( t, x, y_n, z_n) \rightarrow f( t, x, y, z)$ as $n\rightarrow
+\infty$.
\end{itemize}
According to assumption $(iv)$ and $(v)$, it follows from the works of Bahlali et al. \cite{4} or Aman et al. \cite{1} without the Neumann term, for each $(t,x)\in [0,T]\times \mathbb{R}^d$ and every $n\geq C$, RBDSDE associated to  $(l(X_{T}^{t,x}),\underline{f}_{n}, g, h)$ (resp. to $(l(X_{T}^{t,x}),\overline{f}_{n}, g, h)$) has a unique solution $(\underline{Y}^{t,x,n},\underline{Z}^{t,x,n},\underline{K}^{t,x,n})$ (resp. $(\overline{Y}^{t,x,n},\overline{Z}^{t,x,n},\overline{K}^{t,x,n})$). Moreover, it follows again from Bahlali et al. \cite{4} that $(\underline{Y}^{t,x,n},\underline{Z}^{t,x,n},\underline{K}^{t,x,n})$ (resp. $(\overline{Y}^{t,x,n},\overline{Z}^{t,x,n},\overline{K}^{t,x,n})$) converges to the minimal solution (resp. maximal solution) of RBDSDE \eqref{5eddsr}. Setting 
\begin{eqnarray}\label{v1}
\overline{u}_{n}(t,x)=\overline{Y}^{t,x,n}_t,\\\nonumber\\
\underline{u}_n(t,x)=\underline{Y}^{t,x,n}_t,\label{v2}	
\end{eqnarray}
 it follows also from Aman et al \cite{2} that $\overline{u}_{n}$ (resp. $\underline{u}_{n}$) belongs in $\mathcal{C}^{0,0}(\textbf{\textit{F}}_{}^{B},[0, T] \times \R^{d}; \R)$. On the other hand, according to $(ii)$ above and comparison Theorem 3.2 in \cite{4}, the sequence of random field $\overline{u}_{n}$ (resp. $\underline{u}_{n}$) is non-decreasing (resp.
non-increasing). Moreover, for $(t,x)\in [0,T]\times \R^d$ $\overline{u}_{n}$ (resp. $\underline{u}_{n}$) converge to $\overline{u}(t,x)=\overline{Y}^{t,x}_t$ (resp. $\underline{u}(t,x)=\underline{Y}^{t,x}_t$) which is low semi-continuous (resp. upper semi-continuous). Since in the section 2.2 we prove that RBDSDE \eqref{1.6} has a unique solution $(Y^{t,x},Z^{t,x},K^{t,x})$, then $\underline{Y}^{t,x}_t=Y^{t,x}=\overline{Y}^{t,x}_t$. Finally $\overline{u}=\underline{u}=u$. is both lower and upper semi-continuous, i.e., $u\in \mathcal{C}^{0,0}(\textbf{\textit{F}}_{}^{B},[0, T] \times \R^{d}; \R)$.
\end{proof}
The main result of this section is the following
\begin{theorem}\label{5ths1C}
Under conditions $({\bf A1})$-$({\bf A5})$, the random field
$u\in \mathcal{C}^{0,0}(\textbf{\textit{F}}_{}^{B},[0, T] \times \R^{q}; \R)$ defined by \eqref{ao} is a stochastic viscosity solution for SPDE \eqref{5edps}.
\end{theorem}

\begin{proof}
Since for all $(t,x)\in [0,T]\times\R^d,\, u(t,x)=Y^{t,x}$, we have $u(T,x)=l(x)$. Moreover for all $(\tau,\xi)\in\mathcal{M}_{0,T}^{B}\times L^{0}(\mathcal{F}_{\tau}^{B}, \R^{q})$,
\begin{eqnarray}\label{ce2}
u(\tau,\xi)=Y_{\tau}^{\tau,\xi}\geq h(\tau,\xi)\; \; \P_2-a.s.
\end{eqnarray}
Now, it remains to show that $u$ satisfies \eqref{df1} and \eqref{df1}.

For this purpose, For every $n\geq C$, let define $\underline{u}_{n}: \Omega_2\times[0,T]\times \mathbb{R}^q\mapsto \R$ by  \eqref{v2}. Then, with the same argument as above, the sequence of random field $\underline{u}_{n}$ converges to random field $u$ defined by \eqref{ao}. Moreover, using
Theorem 3.1 (without the Neumann term) in Aman et al., \cite{2}, $\underline{u}_n$ is a stochastic viscosity solution of the parabolic SPDE associated to the data $(\underline{f}_n,g,h,l)$,
\begin{eqnarray}\label{ama}
\left\{
\begin{array}{ll}
&\displaystyle \min\Big\{\underline{u}_{n}(t,x)- h(t,x)\;,\;\frac{\partial \underline{u}_{n}}{\partial t}(t,x)+\mathcal{L}\underline{u}_{n}(t,x)+\underline{f}_{n}(t,x,\underline{u}_{n}(t,x),(\sigma^*D\underline{u}_{n})(t,x))\\\\
&\displaystyle\hspace{3.5cm}+g(t,x,\underline{u}_{n}(t,x))\overleftarrow{\dot{B}_{t}}\Big\},\; \; \; \; (t,x)\in[0,T]\times \mathbb{R}^q, \\\\
&\displaystyle \underline{u}_{n}(T,x)= l(x)\hbox{,}\; \; \; \;x\in\mathbb{R}^q.
\end{array}
\right.
\end{eqnarray}
For $\omega_2\in \Omega$ be fixed such that
\begin{eqnarray*}
\underline{u}_n(\omega_2,t,x)\rightarrow u(\omega_2,t,x)	\;\; \mbox{as}\;\; n\rightarrow +\infty,
\end{eqnarray*}
let us consider $(\tau,\xi,\varphi)\in\mathcal{M}_{0,T}^{B}\times L^{0}(\mathcal{F}_{\tau}^{B}, \R^{q})\times C^{1,2}(\mathcal{F}_{\tau}^{B}, [0, T] \times \R^{q};\R)$ such that $0<\tau(\omega_{2})<T$
$$u(\omega_{2},t,x)-\psi(\omega_{2}, t,x)\leq 0= u(\tau(\omega_{2}),\xi(\omega_{2}))-\psi(\tau(\omega_{2}),\xi(\omega_{2})),$$
for all $(t,x)$ in a neighborhood $\mathcal{V}(\tau(\omega_{2}),\xi(\omega_{2}))$ of $(\tau(\omega_{2}),\xi(\omega_{2}))$, where $\psi(t,x)\overset{}{=}\eta(t,x,\varphi(t,x))$. %, $\P^{2}-a.s.$,\ on.
From Example 8.2 in El Karoui et al. (1997) and
Lemma 6.1 in Crandall et al. (1992), there exists sequence $(\tau_{j}(\omega_{2}),\xi_{j}(\omega_{2}),\varphi_{j}(\omega_{2}))_{j\geq1}\in[0,T]\times \R^{q}\times C^{1,2}([0, T] \times \R^{q};\R)$ such that\\ $n_{j}\longrightarrow+\infty,\; \tau_{j}(\omega_{2})\longrightarrow\tau_{}(\omega_{2}),\; \xi_{j}(\omega_{2})\longrightarrow\xi_{}(\omega_{2}),\; \varphi_{j}(\omega_{2})\longrightarrow\varphi_{}(\omega_{2})$ and $$\underline{u}_{n_{j}}(\omega_{2},t,x)-\psi_{j}(\omega_{2},t,x)\leq \underline{u}_{n_{j}}(\tau_{j}(\omega_{2}),\xi_{j}(\omega_{2}))-\psi_{j}(\tau_{j}(\omega_{2}),\xi_{j}(\omega_{2})),$$
for all $(t,x)$ in a neighborhood $\mathcal{V}(\tau_{j}(\omega_{2}),\xi_{j}(\omega_{2}))\subset\mathcal{V}(\tau(\omega_{2}),\xi(\omega_{2}))$ and a suitable subsequence $(\underline{u}_{n_{j}})_{j\geq1}$, where $\psi_{j}(t,x)\overset{}{=}\eta(t,x,\varphi_{j}(t,x))$.  \\ From \eqref{cp2} and \eqref{ce2}, it follows that for $j$ large enough $\underline{u}_{n_{j}}(\tau_{j},\xi_{j})-h(\tau_{j},\xi_{j})\geq 0 \; \P_2-a.s.$.\\ %on $\left\{0<\tau_j<T\right\}$.
Now, using the fact
that $\underline{u}_{n_{j}}$ is a stochastic viscosity solution for SPDE$(\underline{f}_{n_{j}},g,h,l)$, we obtain
$\P_{2}-a.s.$,\ on $\left\{0<\tau_j<T\right\}$,
\begin{eqnarray}\label{5anC}
\mathcal{A}_{\underline{f}_{n_{j}},g}(\psi_j(\tau_j,\xi_j))- D_y\psi_j(\tau_j,\xi_j)D_t\varphi_j(\tau_j,\xi_j)\leq0.
\end{eqnarray}
From the properties of $\eta$, $\psi_j(\tau_j,\xi_j)\overset{}{=}\eta(\tau_j,\xi_j,\varphi_j(\tau_j,\xi_j))$ converges to  $\psi(\tau,\xi)\overset{}{=}\eta(\tau,\xi,\varphi(\tau,\xi))$. %in $[0, T] \times \R^{q}$.
\\Moreover, from the properties of $\underline{f}_{n_{j}}$,
\begin{eqnarray*}
\mathcal{A}_{\underline{f}_{n_{j}},g}(\psi_j(\tau_j,\xi_j))&=&-\mathcal{L}\psi_j(\tau_j,\xi_j)-
\underline{f}_{n_{j}}\big(\tau_j,\xi_j,\psi_{j}(\tau_j,\xi_j),\sigma^*(\tau_j,\xi_j)D_x\psi_{j}(\tau_j,\xi_j)\big)\\&&+\frac{1}{2}\langle g,D_yg\rangle(\tau_j,\xi_j).
\end{eqnarray*}
converges to
\begin{eqnarray*}
\mathcal{A}_{f,g}(\psi(\tau,\xi))&=&-\mathcal{L}\psi(\tau,\xi)-
f\big(\tau,\xi,\psi(\tau,\xi),\sigma^*(\tau,\xi)D_x\psi(\tau,\xi)\big)+\frac{1}{2}\langle g,D_yg\rangle(\tau,\xi).
\end{eqnarray*}
%.
Hence, taking the limit as $j\longrightarrow\infty$ in \eqref{5anC}, we obtain
\begin{eqnarray*}\label{}
\mathcal{A}_{f,g}(\psi(\tau,\xi))-D_y\psi(\tau,\xi)D_t\varphi(\tau,\xi)\leq 0.
\end{eqnarray*}
and we get that $u$ %$u\in \mathcal{C}^{0,0}(\textbf{\textit{F}}_{}^{B},[0, T] \times \R^{q}; \R)$
is a stochastic viscosity subsolution for the SPDE$(f,g,h,l)$. Similarly, we prove that $u$ is a stochastic viscosity supersolution for the SPDE$(f,g,h,l)$. So we conclude that $u$ is a stochastic viscosity solution for the SPDE$(f,g,h,l)$.

\end{proof}
\begin{remark}
Replace $\underline{u}_{n}$ by $\overline{u}_{n}$, we obtain with some adaptation the same conclusion.
\end{remark}

\end{document}